\documentclass{article}
\usepackage{amssymb}
\usepackage{amsfonts}
\usepackage{amsmath}

\newtheorem{theorem}{Theorem}[section]

\newtheorem{corollary}{Corollary}[section]

\newtheorem{lemma}{Lemma}[section]

\newenvironment{proof}[1][Proof]{\noindent\textbf{#1.} }{\ \rule{0.5em}{0.5em}}
\input{tcilatex}

\begin{document}

\title{Iterated binomial transform of the $k$-Lucas sequence}
\author{\texttt{Nazmiye Yilmaz} \: and \: \texttt{Necati Taskara}}
\date{Department of Mathematics, Faculty of Science,\\
Selcuk University, Campus, 42075, Konya - Turkey \\
[0.3cm] \textit{nzyilmaz@selcuk.edu.tr} \: and \: \textit{%
ntaskara@selcuk.edu.tr}}
\maketitle

\begin{abstract}
In this study, we apply "$r$" times the binomial transform to $k$-Lucas
sequence. Also, the Binet formula, summation, generating function of this
transform are found using recurrence relation. Finally, we give the
properties of iterated binomial transform with classical Lucas sequence.

\textbf{Keywords}: $k$-Lucas sequence, iterated binomial transform, Pell
sequence.
\end{abstract}

\section{Introduction and Preliminaries}

There are so many studies in the literature that concern about the special
number sequences such as Fibonacci, Lucas and generalized Fibonacci anad
Lucas numbers (see, for example [1-4], and the references cited therein). In
Fibonacci and Lucas numbers, there clearly exists the term Golden ratio
which is defined as the ratio of two consecutive of these numbers that
converges to $\alpha =\frac{1+\sqrt{5}}{2}$. It is also clear that the ratio
has so many applications in, specially, Physics, Engineering, Architecture,
etc.\cite{Marek-Crnjac,Marek-Crnjac1}.

For $n\geq 1$, $k$-Lucas sequence is defined by the recursive equation:%
\begin{equation}
L_{k,n+1}=kL_{k,n}+L_{k,n-1},\ \ \ \ \ L_{k,0}=2\text{ and}\ L_{k,1}=k.
\label{1.02}
\end{equation}

In addition, some matrices based transforms can be introduced for a given
sequence. Binomial transform is one of these transforms and there are also
other ones such as rising and falling binomial transforms(see [7-14]). Given
an integer sequence $X=\left\{ x_{0},x_{1},x_{2},\ldots \right\} ,$ the
binomial transform $B$ of the sequence $X$, $B\left( X\right) =\left\{
b_{n}\right\} ,$\ is given by%
\begin{equation*}
b_{n}=\sum_{i=0}^{n}\binom{n}{i}x_{i}.
\end{equation*}%
\ In [12], authors gave the application of the several class of transforms
to the $k$-Lucas sequence. For example, for $n\geq 1,\ $authors obtained
recurrence relation of the binomial transform for $k$-Lucas sequence%
\begin{equation*}
b_{k,n+1}=\left( 2+k\right) b_{k,n}-kb_{k,n-1},\ \ b_{k,0}=2\text{ and }%
b_{k,1}=k+2.
\end{equation*}%
Falcon [13] studied the iterated application of the some Binomial transforms
to the $k$-Fibonacci sequence. For example, author obtained recurrence
relation of the iterated binomial transform for $k$-Fibonacci sequence%
\begin{equation*}
c_{k,n+1}^{\left( r\right) }=\left( 2r+k\right) c_{k,n}^{\left( r\right)
}-\left( r^{2}+kr-1\right) c_{k,n-1}^{\left( r\right) },\ \ c_{k,0}^{\left(
r\right) }=0\text{ and }c_{k,1}^{\left( r\right) }=1.
\end{equation*}

Motivated by [13], the goal of this paper is to apply iteratly the binomial
transform to the $k$-Lucas sequence. Also, the properties of this transform
are found by recurrence relation. Finally, it is illustrated the relation
between of this transform and the iterated binomial transform of $k$%
-Fibonacci sequence by deriving new formulas.

\section{Iterated binomial transform of $k$-Lucas sequences}

In this section, we will mainly focus on iterated binomial transforms of $k$%
-Lucas sequences to get some important results. In fact, we will also
present the recurrence relation, Binet formula, summation, generating
function of the transform and relationships betweeen of the transform and
iterated binomial transform of $k$-Fibonacci sequence.

The iterated binomial transform of the $k$-Lucas sequences is demonstrated
by $B_{k}^{\left( r\right) }=\left\{ b_{k,n}^{\left( r\right) }\right\} ,$
where $b_{k,n}^{\left( r\right) }$ is obtained by applying $"r"$ times the
binomial transform to $k$-Lucas sequence. It is obvious that $%
b_{k,0}^{\left( r\right) }=2$ and $b_{k,1}^{\left( r\right) }=2r+k.$

\bigskip

The following lemma will be key of the proof of the next theorems.

\begin{lemma}
\label{lem1}For $n\geq 0$ and $r\geq 1,$ the following equality is hold:%
\begin{equation*}
b_{k,n+1}^{\left( r\right) }=b_{k,n}^{\left( r\right) }+\sum_{j=0}^{n}\binom{%
n}{j}b_{k,j+1}^{\left( r-1\right) }.
\end{equation*}
\end{lemma}

\begin{proof}
By using definition of binomial transform and the well known binomial
equality 
\begin{equation*}
\binom{n+1}{i}=\binom{n}{i}+\binom{n}{i-1},
\end{equation*}%
we obtain%
\begin{eqnarray*}
b_{k,n+1}^{\left( r\right) } &=&\sum_{j=0}^{n+1}\binom{n+1}{j}%
b_{k,j}^{\left( r-1\right) } \\
&=&\sum_{j=1}^{n+1}\binom{n+1}{j}b_{k,j}^{\left( r-1\right)
}+b_{k,0}^{\left( r-1\right) }
\end{eqnarray*}%
\begin{eqnarray*}
b_{k,n+1}^{\left( r\right) } &=&\sum_{j=1}^{n+1}\binom{n}{j}b_{k,j}^{\left(
r-1\right) }+\sum_{j=1}^{n+1}\binom{n}{j-1}b_{k,j}^{\left( r-1\right)
}+b_{k,0}^{\left( r-1\right) } \\
&=&\sum_{j=0}^{n+1}\binom{n}{j}b_{k,j}^{\left( r-1\right) }+\sum_{j=0}^{n+1}%
\binom{n}{j-1}b_{k,j}^{\left( r-1\right) } \\
&=&\sum_{j=0}^{n}\binom{n}{j}b_{k,j}^{\left( r-1\right) }+\sum_{j=-1}^{n}%
\binom{n}{j}b_{k,j+1}^{\left( r-1\right) } \\
&=&b_{k,n}^{\left( r\right) }+\sum_{j=0}^{n}\binom{n}{j}b_{k,j+1}^{\left(
r-1\right) }
\end{eqnarray*}%
which is desired result.
\end{proof}

In \cite{BhadouriaJhalaSingh}, the authors obtained the following equality
for binomial transform of $k$-Lucas sequences. However, in here, we obtain
the equality in terms of iterated binomial transform of the $k$-Lucas
sequences as a consequence of Lemma 2.1. To do that we take $r=1$ in Lemma
2.1: 
\begin{equation*}
b_{k,n+1}=b_{k,n}+\sum_{j=0}^{n}\binom{n}{j}L_{k,j+1}.
\end{equation*}

\begin{theorem}
\label{teo1}For $n\geq 0$ and $r\geq 1,$ the recurrence relation of sequence 
$\left\{ b_{k,n}^{\left( r\right) }\right\} $ is%
\begin{equation}
b_{k,n+1}^{\left( r\right) }=\left( 2r+k\right) b_{k,n}^{\left( r\right)
}-\left( r^{2}+kr-1\right) b_{k,n-1}^{\left( r\right) },  \label{2.1}
\end{equation}%
with initial conditions $b_{k,0}^{\left( r\right) }=2$ and $b_{k,1}^{\left(
r\right) }=2r+k.$
\end{theorem}

\begin{proof}
The proof will be done by induction steps on $r$ and $n$.

First of all, for $r=1,$ from the Equality 2.2 in \cite{BhadouriaJhalaSingh}%
, it is true $b_{k,n+1}=\left( 2+k\right) b_{k,n}-kb_{k,n-1}.$

Let us consider definition of iterated binomial transform, then we have 
\begin{equation*}
b_{k,2}^{\left( r\right) }=k^{2}+2rk+2r^{2}+2.
\end{equation*}%
The initial conditions are%
\begin{equation*}
b_{k,0}^{\left( r\right) }=2\text{ and }b_{k,1}^{\left( r\right) }=2r+k.
\end{equation*}%
Hence, for $n=1,$ the equality (\ref{2.1}) is true, that is $b_{k,2}^{\left(
r\right) }=\left( 2r+k\right) b_{k,1}^{\left( r\right) }-\left(
r^{2}+kr-1\right) b_{k,0}^{\left( r\right) }.$

Actually, by assuming the equation in (\ref{2.1}) holds for all $(r-1,n)$
and $(r,n-1),$ that is,%
\begin{equation*}
b_{k,n+1}^{\left( r-1\right) }=\left( 2r-2+k\right) b_{k,n}^{\left(
r-1\right) }-\left( \left( r-1\right) ^{2}+k\left( r-1\right) -1\right)
b_{k,n-1}^{\left( r-1\right) },
\end{equation*}%
and%
\begin{equation*}
b_{k,n}^{\left( r\right) }=\left( 2r+k\right) b_{k,n-1}^{\left( r\right)
}-\left( r^{2}+kr-1\right) b_{k,n-2}^{\left( r\right) }.
\end{equation*}%
Now, by taking account Lemma 2.1, we obtain%
\begin{eqnarray*}
b_{k,n+1}^{\left( r\right) } &=&b_{k,n}^{\left( r\right) }+\sum_{j=0}^{n}%
\binom{n}{j}b_{k,j+1}^{\left( r-1\right) } \\
&=&\sum_{j=0}^{n}\binom{n}{j}b_{k,j}^{\left( r-1\right) }+\sum_{j=0}^{n}%
\binom{n}{j}b_{k,j+1}^{\left( r-1\right) } \\
&=&\sum_{j=1}^{n}\binom{n}{j}\left( b_{k,j}^{\left( r-1\right)
}+b_{k,j+1}^{\left( r-1\right) }\right) +b_{k,0}^{\left( r-1\right)
}+b_{k,1}^{\left( r-1\right) }.
\end{eqnarray*}%
By reconsidering our assumption, we write%
\begin{eqnarray*}
b_{k,n+1}^{\left( r\right) } &=&\sum_{j=1}^{n}\binom{n}{j}\left(
b_{k,j}^{\left( r-1\right) }+\left( 2r-2+k\right) b_{k,j}^{\left( r-1\right)
}-\left( r^{2}-2r+kr-k\right) b_{k,j-1}^{\left( r-1\right) }\right)
+b_{k,0}^{\left( r-1\right) }+b_{k,1}^{\left( r-1\right) } \\
&=&\left( 2r+k-1\right) \sum_{j=1}^{n}\binom{n}{j}b_{k,j}^{\left( r-1\right)
}-\left( r^{2}-2r+kr-k\right) \sum_{j=1}^{n}\binom{n}{j}b_{k,j-1}^{\left(
r-1\right) }+b_{k,0}^{\left( r-1\right) }+b_{k,1}^{\left( r-1\right) } \\
&=&\left( 2r+k-1\right) \sum_{j=0}^{n}\binom{n}{j}b_{k,j}^{\left( r-1\right)
}-\left( r^{2}-2r+kr-k\right) \sum_{j=1}^{n}\binom{n}{j}b_{k,j-1}^{\left(
r-1\right) }+b_{k,0}^{\left( r-1\right) }+b_{k,1}^{\left( r-1\right) } \\
&&-\left( 2r+k-1\right) b_{k,0}^{\left( r-1\right) } \\
&=&\left( 2r+k-1\right) b_{k,n}^{\left( r\right) }-\left(
r^{2}-2r+kr-k\right) \sum_{j=1}^{n}\binom{n}{j}b_{k,j-1}^{\left( r-1\right)
}+\left( 2-2r-k\right) b_{k,0}^{\left( r-1\right) }+b_{k,1}^{\left(
r-1\right) }.
\end{eqnarray*}%
Then we have%
\begin{equation}
b_{k,n+1}^{\left( r\right) }-\left( 2r+k-1\right) b_{k,n}^{\left( r\right)
}=-\left( r^{2}-2r+kr-k\right) \sum_{j=1}^{n}\binom{n}{j}b_{k,j-1}^{\left(
r-1\right) }+4-2r-k.  \label{2.2}
\end{equation}%
By taking $n\rightarrow n-1,$ it is%
\begin{eqnarray*}
b_{k,n}^{\left( r\right) } &=&\left( 2r+k-1\right) b_{k,n-1}^{\left(
r\right) }-\left( r^{2}-2r+kr-k\right) \sum_{j=1}^{n-1}\binom{n-1}{j}%
b_{k,j-1}^{\left( r-1\right) }+4-2r-k \\
&=&\left( 2r+k-1\right) b_{k,n-1}^{\left( r\right) }-\left(
r^{2}-2r+kr-k\right) \sum_{j=1}^{n}\left[ \binom{n}{j}-\binom{n-1}{j-1}%
\right] b_{k,j-1}^{\left( r-1\right) }+4-2r-k \\
&=&\left( 2r+k-1\right) b_{k,n-1}^{\left( r\right) }-\left(
r^{2}-2r+kr-k\right) \sum_{j=1}^{n}\binom{n}{j}b_{k,j-1}^{\left( r-1\right) }
\\
&&+\left( r^{2}-2r+kr-k\right) \sum_{j=1}^{n}\binom{n-1}{j-1}%
b_{k,j-1}^{\left( r-1\right) }+4-2r-k
\end{eqnarray*}%
\begin{eqnarray*}
b_{k,n}^{\left( r\right) } &=&\left( 2r+k-1\right) b_{k,n-1}^{\left(
r\right) }-\left( r^{2}-2r+kr-k\right) \sum_{j=1}^{n}\binom{n}{j}%
b_{k,j-1}^{\left( r-1\right) } \\
&&+\left( r^{2}-2r+kr-k\right) \sum_{j=0}^{n-1}\binom{n-1}{j}b_{k,j}^{\left(
r-1\right) }+4-2r-k \\
&=&\left( 2r+k-1\right) b_{k,n-1}^{\left( r\right) }-\left(
r^{2}-2r+kr-k\right) \sum_{j=1}^{n}\binom{n}{j}b_{k,j-1}^{\left( r-1\right) }
\\
&&+\left( r^{2}-2r+kr-k\right) b_{k,n-1}^{\left( r\right) }+4-2r-k \\
&=&\left( r^{2}+kr-1\right) b_{k,n-1}^{\left( r\right) }-\left(
r^{2}-2r+kr-k\right) \sum_{j=1}^{n}\binom{n}{j}b_{k,j-1}^{\left( r-1\right)
}+4-2r-k.
\end{eqnarray*}%
Hence, we have%
\begin{equation*}
b_{k,n}^{\left( r\right) }-\left( r^{2}+kr-1\right) b_{k,n-1}^{\left(
r\right) }=-\left( r^{2}-2r+kr-k\right) \sum_{j=1}^{n}\binom{n}{j}%
b_{k,j-1}^{\left( r-1\right) }+4-2r-k.
\end{equation*}%
If last expression put in place in the equality (\ref{2.2}), then we get%
\begin{eqnarray*}
b_{k,n+1}^{\left( r\right) } &=&\left( 2r+k-1\right) b_{k,n}^{\left(
r\right) }+b_{k,n}^{\left( r\right) }-\left( r^{2}+kr-1\right)
b_{k,n-1}^{\left( r\right) } \\
&=&\left( 2r+k\right) b_{k,n}^{\left( r\right) }-\left( r^{2}+kr-1\right)
b_{k,n-1}^{\left( r\right) }
\end{eqnarray*}%
which is completed the proof of this theorem.
\end{proof}

The characteristic equation of sequence $\left\{ b_{k,n}^{\left( r\right)
}\right\} $ in (\ref{2.1}) is \linebreak $\lambda ^{2}-\left( 2r+k\right)
\lambda +r^{2}+kr-1=0.$ Let be $\lambda _{1}$ and $\lambda _{2}$ the roots
of this equation. Then, Binet's formulas of sequence $\left\{
b_{k,n}^{\left( r\right) }\right\} $ can be expressed as%
\begin{equation}
b_{k,n}^{\left( r\right) }=\left( \frac{k+\sqrt{k^{2}+4}}{2}+r\right)
^{n}+\left( \frac{k-\sqrt{k^{2}+4}}{2}+r\right) ^{n}.  \label{2.30}
\end{equation}%
\qquad \qquad \qquad \qquad

In here, we obtain the equalities given in [12] in terms of iterated
binomial transform of the $k$-Lucas sequences as a consequence of Theorem
2.1. To do that we take $r=1$ in Theorem 2.1 and \ the equation (\ref{2.30}%
): 
\begin{equation*}
b_{k,n+1}=\left( 2+k\right) b_{k,n}-kb_{k,n-1},
\end{equation*}%
and%
\begin{equation*}
b_{k,n}=\left( \frac{k+2+\sqrt{k^{2}+4}}{2}\right) ^{n}+\left( \frac{k+2-%
\sqrt{k^{2}+4}}{2}\right) ^{n}.
\end{equation*}

Now, we give the sum of iterated binomial transform for $k$-Lucas sequences.

\begin{theorem}
Sum of $\ $sequence $\left\{ b_{k,n}^{\left( r\right) }\right\} $ is%
\begin{equation*}
\sum_{i=0}^{n-1}b_{k,i}^{\left( r\right) }=\frac{\left( r^{2}+kr-1\right)
b_{k,n-1}^{\left( r\right) }-b_{k,n}^{\left( r\right) }-k-2r+2}{r^{2}+kr-k-2r%
}.
\end{equation*}
\end{theorem}

\begin{proof}
By considering equation (\ref{2.30}), we have%
\begin{equation*}
\sum\limits_{i=0}^{n-1}b_{k,i}^{\left( r\right)
}=\sum\limits_{i=0}^{n-1}\left( \lambda _{1}^{i}+\lambda _{2}^{i}\right) .
\end{equation*}%
Then we obtain%
\begin{equation*}
\sum\limits_{i=0}^{n-1}b_{k,i}^{\left( r\right) }=\left( \frac{\lambda
_{1}^{n}-1}{\lambda _{1}-1}\right) +\left( \frac{\lambda _{2}^{n}-1}{\lambda
_{2}-1}\right) .
\end{equation*}%
Afterward, by taking account equations $\lambda _{1}.\lambda _{2}=r^{2}+kr-1$
and $\lambda _{1}+\lambda _{2}=k+2r,$\ we conclude 
\begin{equation*}
\sum\limits_{i=0}^{n-1}b_{k,i}^{\left( r\right) }=\frac{\left(
r^{2}+kr-1\right) b_{k,n-1}^{\left( r\right) }-b_{k,n}^{\left( r\right)
}-k-2r+2}{r^{2}+kr-k-2r}.
\end{equation*}
\end{proof}

Note that, if we take $r=1$ in Theorem 2.2, we obtain the summation of
binomial transform for $k$-Lucas sequence:%
\begin{equation*}
\sum\limits_{i=0}^{n-1}b_{k,i}=b_{k,n}-kb_{k,n-1}+k
\end{equation*}

\begin{theorem}
\label{teo2}The generating function of the iterated binomial transform for $%
\left\{ L_{k,n}\right\} $\ is%
\begin{equation*}
\sum\limits_{i=0}^{\infty }b_{k,i}^{\left( r\right) }x^{i}=\frac{2-\left(
2r+k\right) x}{1-\left( 2r+k\right) x+\left( r^{2}+kr-1\right) x^{2}}.
\end{equation*}
\end{theorem}

\begin{proof}
Assume that $b\left( k,x,r\right) =\sum\limits_{i=0}^{\infty
}b_{k,i}^{\left( r\right) }x^{i}$ is the generating function of the iterated
binomial transform for $\left\{ L_{k,n}\right\} $. From Theorem \ref{teo1},
we obtain 
\begin{eqnarray*}
b\left( k,x,r\right) &=&b_{k,0}^{\left( r\right) }+b_{k,1}^{\left( r\right)
}x+\sum\limits_{i=2}^{\infty }\left( \left( 2r+k\right) b_{k,i-1}^{\left(
r\right) }-\left( r^{2}+kr-1\right) b_{k,i-2}^{\left( r\right) }\right) x^{i}
\\
&=&b_{k,0}^{\left( r\right) }+b_{k,1}^{\left( r\right) }x-\left( 2r+k\right)
b_{k,0}^{\left( r\right) }x+\left( 2r+k\right) x\sum\limits_{i=0}^{\infty
}b_{k,i}^{\left( r\right) }x^{i} \\
&&-\left( r^{2}+kr-1\right) x^{2}\sum\limits_{i=0}^{\infty }b_{k,i}^{\left(
r\right) }x^{i} \\
&=&b_{k,0}^{\left( r\right) }+\left( b_{k,1}^{\left( r\right) }-\left(
2r+k\right) b_{k,0}^{\left( r\right) }\right) x+\left( 2r+k\right) xb\left(
k,x,r\right) \\
&&-\left( r^{2}+kr-1\right) x^{2}\,b\left( k,x,r\right) .
\end{eqnarray*}

Now rearrangement the equation implies that 
\begin{equation*}
b\left( k,x,r\right) =\frac{b_{k,0}^{\left( r\right) }+\left(
b_{k,1}^{\left( r\right) }-\left( 2r+k\right) b_{k,0}^{\left( r\right)
}\right) x}{1-\left( 2r+k\right) x+\left( r^{2}+kr-1\right) x^{2}},
\end{equation*}%
which equal to the $\sum\limits_{i=0}^{\infty }b_{k,i}^{\left( r\right)
}x^{i}$ in theorem. Hence the result.
\end{proof}

In here, we obtain the generating function given in [12] in terms of
iterated binomial transform of the $k$-Lucas sequences as a consequence of
Theorem 2.3. To do that we take $r=1$ in Theorem 2.3:%
\begin{equation*}
\sum\limits_{i=0}^{\infty }b_{k,i}x^{i}=\frac{2-\left( 2+k\right) x}{%
1-\left( 2+k\right) x+kx^{2}}.
\end{equation*}

In the following theorem, we present the relationship between the iterated
binomial transform of $k$-Lucas sequence and iterated binomial transform of $%
\ k$-Fibonacci sequence.

\begin{theorem}
\label{teo3}For $n>0$, the relationship of between the transforms $\left\{
b_{k,n}^{\left( r\right) }\right\} $ and $\left\{ c_{k,n}^{\left( r\right)
}\right\} $ is illustrated by following way:%
\begin{equation}
b_{k,n}^{\left( r\right) }=c_{k,n+1}^{\left( r\right) }-\left(
r^{2}+kr-1\right) c_{k,n-1}^{\left( r\right) },  \label{2.5}
\end{equation}%
where $b_{k,n}^{\left( r\right) }$ is the iterated binomial transform of $k$%
-Lucas sequence and $c_{k,n}^{\left( r\right) }$ is the iterated binomial
transform of $k$-Fibonacci sequence.
\end{theorem}

\begin{proof}
By using the equality in (\ref{2.5}), let be%
\begin{equation*}
b_{k,n}^{\left( r\right) }=Xc_{k,n+1}^{\left( r\right) }+Yc_{k,n-1}^{\left(
r\right) }.
\end{equation*}%
If we take $n=1$ and $2$, we have the system%
\begin{equation*}
\left\{ 
\begin{array}{c}
b_{k,1}^{\left( r\right) }=Xc_{k,2}^{\left( r\right) }+Yc_{k,0}^{\left(
r\right) }, \\ 
b_{k,2}^{\left( r\right) }=Xc_{k,3}^{\left( r\right) }+Yc_{k,1}^{\left(
r\right) }.%
\end{array}%
\right.
\end{equation*}%
By considering definition of the iterated binomial transforms for $k$-Lucas, 
$k$-Fibonacci sequence and Cramer rule for the system, we obtain%
\begin{equation*}
\left\{ 
\begin{array}{c}
2r+k=\left( 2r+k\right) X, \\ 
k^{2}+2rk+2r^{2}+2=\left( 3r^{2}+3rk+k^{2}+1\right) X+Y%
\end{array}%
\right.
\end{equation*}%
and%
\begin{equation*}
X=1\text{ and}\ Y=-\left( r^{2}+kr-1\right)
\end{equation*}%
which is completed the proof of this theorem.
\end{proof}

Note that, if we take $r=1$ in Theorem 2.4, we obtain the relationship of
between the binomial transform for $k$-Lucas sequence and the binomial
transform for $k$-Fibonacci sequence:%
\begin{equation*}
b_{k,n}=c_{k,n+1}-kc_{k,n-1}.
\end{equation*}

\begin{corollary}
We should note that choosing $k=1$ in the all results of section 2, it is
actually obtained some properties of the iterated binomial transform for
classical Lucas sequence such that the recurrence relation, Binet formula,
summation, generating function and relationship of between binomial
transforms for Fibonacci and Lucas sequences.
\end{corollary}

\begin{corollary}
We should note that choosing $k=2$ in the all results of section 2, it is
actually obtained some properties of the iterated binomial transform for
classical Pell-Lucas sequence such that the recurrence relation, Binet
formula, summation, generating function and relationship of between binomial
transforms for Pell and Pell-Lucas sequences.
\end{corollary}

\end{document}